\documentclass[12pt]{amsart}\usepackage{amsmath,amssymb,times,color}
\usepackage{graphicx}
\usepackage{dsfont}
\definecolor{webred}{rgb}{0.75,0,0}
\definecolor{webgreen}{rgb}{0,0.75,0}
\usepackage[citecolor=webgreen,colorlinks=true,linkcolor=webred]{hyperref}
\usepackage[english]{babel}

\numberwithin{equation}{section}

\date{\today}

\makeindex \makeatletter
\def\captionof#1#2{{\def\@captype{#1}#2}}
\makeatother

\newtheorem{definition}{D\'efinition}[section]
\newtheorem{proposition}{Proposition}[section]
\newtheorem{remark}{\bf Remark}
\newtheorem{theorem}{\bf Theorem}[section]
\def\bgneqnn{\begin{equation}}
\def\endeqnn{\end{equation}}
\def\bgneqy{\begin{eqnarray}}
\def\endeqy{\end{eqnarray}}
\def\bgneqy*{\begin{eqnarray*}}
\def\endeqy*{\end{eqnarray*}}

\def\text{\mbox}

\def\bgneqy*{\begin{eqnarray*}}
\def\endeqy*{\end{eqnarray*}}

\def\qed{\hfill$\blacksquare$\par\bigskip}

\newcounter{tablegroup}
\newcounter{subtable}[tablegroup]

\newcommand{\handletables}

\begin{document}

\title{Discrete Energy behavior of a damped Timoshenko system.  }



\author{Chebbi Sabrine         \and
        Hamouda Makram$^{*}$} 

\address{$^{*}$ University of Tunis El Manar, Faculty of Sciences of Tunis, Department of Mathematics, Tunis,  Tunisia.}

\maketitle

\begin{abstract}
In this article, we consider a one-dimensional Timoshenko system subject to different types of dissipation (linear and nonlinear dampings). Based on a combination between the finite element  and the finite difference methods, we design a discretization scheme for the different Timoshenko systems under consideration. We first come up with a numerical scheme to the free-undamped Timoshenko system. Then, we adapt this numerical scheme to the corresponding linear and nonlinear damped systems.  Interestingly, this scheme reaches to reproduce the most important properties of the discrete energy. Namely, we show for the discrete energy the positivity, the energy conservation property and the different decay rate profiles.  We  numerically reproduce the known analytical results  established on the decay rate of the energy associated with each type of dissipation.
\end{abstract}

\section{Introduction}\label{intro}
Since the pioneer work of Timoshenko \cite{45} it is now well-known that, in the general theory of  structure's study, the Timoshenko system is a good approximation of beam transverse vibrations. More precisely, these movements can be modeled by a set of two coupled wave equations of the form:
 \begin{equation} \label{0}
     \left \{ \begin{array}{lrl}
 \rho \varphi_{tt} - K(\varphi_x+ \psi)_x=0,  \hspace{2cm } (x,t) \in (0,L)\times \mathbb{R}_+,\\
         I_{\rho}\psi_{tt}-(EI\psi_x)_x+K(\varphi_x+\psi)=0, \hspace{0.3cm }(x,t) \in (0,L)\times \mathbb{R}_+,
          \end{array}\right.
       \end{equation}
where  $t$ is the time, $x$ is the position coordinate along the beam, $\varphi(t,x)$ is the transverse displacement of the beam around an equilibrium state, and $\psi(x,t)$ is the rotation angle of the filament of the beam.  The coefficients $\rho$, $I_{\rho}$, $E$, $I$, and $K$ are, respectively, the density, the polar moment of inertia across a section, the elasticity Young modulus, the moment of inertia across a section and the shear modulus.\\

In the last decades, the study of Timoshenko systems have been attracting the attention of many researchers and the question of the introduction of damping terms  and their influence on the behavior of the solution of  (\ref{0}) are of great interest both for mathematicians and engineers.  Hence, the stabilization of the damped system related to (\ref{0}) is one of the main objective of several studies. In this direction, the present article aims to give some light on the numerical study of the relationship between stabilization and optimality completing thus the theoretical part carried out in \cite{sab}.\\

First, we recall that the exponential stability is  known for Timoshenko systems on bounded domains when   the two linear damping terms $\varphi_t$ and $\psi_t$ are considered in the left-hand sides of the first and second equations of  (\ref{0}), respectively (see e.g. \cite{8}).  In  \cite{5}, the authors considered the one-dimensional system with a linear damping term
as follows:
 \begin{equation} \label{2}
     \left \{ \begin{array}{lrl}
 \rho_1 \varphi_{tt} - k(\varphi_x+ \psi)_x=0,  \hspace{2.2cm } (x,t) \in (0,L)\times \mathbb{R}_+,\\
         \rho_2\psi_{tt}-b\psi_{xx}+k(\varphi_x+\psi)+d \psi_t=0, \hspace{0.18cm }(x,t) \in (0,L)\times \mathbb{R}_+,
       \end{array}\right.
       \end{equation}
and they proved that the solution of the system (\ref{1}) is exponentially stable if and only if the wave speeds are equal $(\frac{k}{\rho_1}=\frac{b}{\rho_2}).$ Such a result has been the common point in several works  with  different types of dissipation \cite{1,5,3,4,2}.  \\

Second, for the nonlinear damped Timoshenko system having a    damping term with no growth assumption at the origin, Alabau-Boussouira \cite{7} established  a general semi-explicit formula for the decay rate of the energy at infinity in the case of equal  speeds of propagation,  and she  proved a polynomial decay in  the case of different speeds of propagation for both linear and nonlinear globally Lipschitz feedbacks. Later, in  \cite{alab}, Alabau-Boussouira also established a strong lower energy estimate for the strong solutions of nonlinearly damped Timoshenko beams and, as an extension of this result, for the nonlinearly damped Timoshenko system with thermoelasticity (see \cite{sab}). From the numerical point of view, the authors in \cite{ali} used a fourth-order finite difference scheme to compute the numerical solutions of the Timoshenko system with thermoelasticity with second sound (coupled with the Cattaneo Law and giving rise to a system with four equations). To this end, the authors in \cite{ali} adapted the method used in \cite{154} and  obtained  the decay rate of the discrete solutions.\\
     
      Recently more attention was given to  the numerical study of the Timoshenko systems (see e.g. \cite{154}).  The present article is mainly concerned with the numerical decay rate of the discrete energy associated with the  solution of the Timoshenko system that we will set subsequently.  We start our study by introducing a spatial discretization using a classical finite element method based on Galerkin approximation.  Then, we continue to design a discretization scheme using  the finite difference method for the time derivative terms and thus we prove  the energy decay rates for the discrete energy which will be, as we will see later, in a good agreement with the results  obtained in the the theoretical context.\\

    Related to the objective of this paper, and to the best of our knowledge, there are few results in the literature concerning the numerical study of one-dimensional Timoshenko systems. We give here a quick overview of the available results in this direction. In  \cite{cons}, aiming to analyze the energy properties of some linearly elastic constant coefficient Timoshenko systems, the author presents a parameterized family of finite-difference schemes  and the emphasis was on the shear deformation and the  rotatory inertia. The proof in \cite{cons} relies on discrete multiplier techniques. Also, in \cite{w}, the numerical exponential decay  rate of the energy of a dissipative Bresse system was obtained using a finite difference method. Nevertheless, it is well known that the Bresse systems are somehow  a generalization of the  Timoshenko beam equations.   For further details about the subject the reader may consult e.g. \cite{Balduzzi,Brandts,Niemi,Papukashvili,Scott} and the references therein.\\

    The remainder of the paper is organized as follows. In section \ref{sec:1} we introduce the numerical scheme using a finite element discretization in space that we apply to the Timoshenko equations taking advantage of the Galerkin approximation method. Next, we employ to the time derivative terms a  finite difference discretization method. We then establish the energy conservation property for the discrete solution of the Timoshenko beam model. In Section \ref{sec:2} we  prove the exponential stability in the presence of a linear damping and the polynomial stability in the case of a nonlinear damping.
Finally, in Section \ref{sec:3}, we discuss the numerical aspect of the energy and we conclude our work.

\section{Energy Conservation Property of the Timoshenko Equations}
\label{sec:1}
 In this section, we present a numerical method related to the solution of the vibrating
Timoshenko beam equations that are given by
 \begin{equation} \label{1}
     \left \{ \begin{array}{lrl}
 \rho_1 \varphi_{tt} - k(\varphi_x+ \psi)_x=0,  \hspace{1.4cm } (x,t) \in (0,L)\times \mathbb{R}_+,\\
         \rho_2\psi_{tt}-b\psi_{xx}+k(\varphi_x+\psi)=0, \hspace{0.4cm }(x,t) \in (0,L)\times \mathbb{R}_+,
          \end{array}\right.
       \end{equation}
and for which we associate the following  initial conditions:
        \begin{equation} \label{8}
     \left \{ \begin{array}{lrl}
\varphi(0,x)=\varphi_0(x), \
\psi(0,x)=\psi_0(x), \hspace{0.5cm} \forall x \in (0,1),\\
\varphi_t(0,x)=\varphi_1(x), \
\psi_t(0,x)=\psi_1(x),\hspace{0.3cm} \forall x \in (0,1),
       \end{array}\right.
       \end{equation}
together with  the Dirichlet  boundary conditions as below
       \begin{equation}\label{BC}
       \varphi=\psi=0, \hspace{0.4cm} \text{at} \ x=0, L.
       \end{equation}

The energy  associated with the solution $U:=(\varphi, \psi)^T$ of (\ref{1})--(\ref{BC}) is defined by
       \begin{equation}
  E(U,t):=\frac{1}{2} \int_0^L  \left(\rho_1\varphi_t^2+\rho_2\psi_t^2+b\psi_x^2+k(\varphi_x+\psi )^2\right)dx.
\end{equation}
In what follows we will simply use $E(t)$ instead of $E(U,t)$ to make the presentation simpler.\\
Now, taking into account the boundary conditions (\ref{BC}), we obtain
\begin{equation}
\frac{dE(t)}{dt}=0, \hspace{0.4cm} \forall \ t\in [0,T],
\end{equation}
which states the conservation of the energy and this can be expressed as follows:
\begin{equation}\label{cosrv}
\begin{tabular}{ll}
$
\displaystyle{ E(t)=E(0) :=\frac{1}{\rho_1} \int_0^L |\varphi_1(x)|^2 dx +\int_0^L|\psi_1(x)|^2 dx +} $\\

$ \hspace{0.9cm} \displaystyle{+\frac{b}{2}\int_0^L |\psi_{x}(0,x)|^2 dx+\frac{k}{2}\int_0^L |\varphi_{x}(0,x)+\psi_0(x)|^2 dx},$  $ \hspace{0.4cm}  \forall \  t\geq 0. $
\end{tabular}
\end{equation}

This energy conservation property implies that the Timoshenko equations are purely conservative.
Therefore, it is important to show that the  numerical solution of the Timoshenko equations  consistently preserves the property (\ref{cosrv}) as well, that is  the  discrete energy  will obey the energy conservation property.
  \subsection{Semi-discrete finite element scheme}
  In order to obtain the discrete energy, we  first consider a numerical scheme using finite element methods and we reproduce numerically the analytical results established  for the   Timoshenko system (\ref{2}) in the case where
 $\rho_1=\rho_2=1$ and $b=k=1$, that is  the speed waves are equal.\\

  For instance we consider  the following undamped Timoshenko problem (\textit{i.e.} (\ref{2}) with $d=0$),
 \begin{equation} \label{22}
     \left \{ \begin{array}{lrl}
\varphi_{tt} -(\varphi_x+ \psi)_x=0,  \hspace{1.4cm } (x,t) \in (0,L)\times \mathbb{R}_+,\\
         \psi_{tt}-(\psi_x)_x+(\varphi_x+\psi)=0, \hspace{0.3cm }(x,t) \in (0,L)\times \mathbb{R}_+.
          \end{array}\right.
       \end{equation}
 Then, we set $u=\varphi_t$, $v=\psi_t$ and  we rewrite the system (\ref{22}) as follows:
   \begin{equation} \label{002}
u_t=\varphi_{xx}+\psi_{x},\hspace{0.4cm} \forall \ x \in (0,L), \  \forall \  t>0,
  \end{equation}
  \begin{equation}\label{122}
v_t=\psi_{xx}-\varphi_{x}-\psi, \hspace{0.4cm} \forall \ x \in (0,L), \  \forall \ t>0,
       \end{equation}
  and as before we associate with \eqref{002}-\eqref{122} the  Dirichlet  boundary conditions \eqref{BC}.

       The variational form is required to approximate solutions with  the finite element methods.
We multiply the equations (\ref{002}) and (\ref{122}) with arbitrary test functions $\textbf{u}$ and $\textbf{v}$, respectively, and we use the following  notation   for convenience,
$$\hspace{0.5cm}(f,g)=\int_0^L f(x) g(x) dx.$$
Integrating over the interval $(0,L)$ and using  an integration by parts, we find
\begin{equation} \label{3}
( u_t ,\textbf{u})=- ( \varphi_{x}, \textbf{u}_x)+(\psi_{x}, \textbf{u}),\hspace{0.6cm} (x,t) \in (0,L)\times \mathbb{R}_+,
 \end{equation}
    \begin{equation} \label{4}
( v_t ,\textbf{v})=-(\psi_{x},\textbf{v}_x)-(\varphi_{x},\textbf{v})-(\psi, \textbf{v}),\hspace{0.2cm}(x,t) \in (0,L)\times \mathbb{R}_+.
       \end{equation}
  Two sets of test functions are required to incorporate the boundary conditions into (\ref{3})-(\ref{4}):
      $$H_1(0,1)=\lbrace \textbf{u} \in \mathcal{C}(0,1): \textbf{u} (0)= \textbf{u}(1)=0 \rbrace.$$

By adding equations (\ref{3}) and (\ref{4}) we end up with the variational form of the problem which can be formulated,   in the product space $H_1(0,1)\times H_1(0,1)$, as follows.\\
 Find $u$, $v$ such that for all $t>0$,  $\textbf{u}\in H_1(0,1)$ and
  $ \textbf{v}\in H_1(0,1)$, we have
      \begin{equation}\label{fv}
     ( u_t ,\textbf{u})+( v_t ,\textbf{v})=(\psi_{x},\textbf{u}) -(\psi_{x},\textbf{v}_x)-(\varphi_{x},\textbf{v})-(\psi, \textbf{v})- ( \varphi_{x},\textbf{u}_x).
     \end{equation}
  Let $N_x \in \mathbb{N}$ and  $h = 1/(N_x+1)$ such that the mesh $x_i = ih$,  $i =\lbrace 0, \cdots ,N_x+1\rbrace $ is a uniform partition of $[0, 1]$. \\
Now, to obtain the  Galerkin approximation of the variational problem (\ref{fv}), let us consider a finite dimensional set of functions $\lbrace w_1, \ldots, w_{N_x} \rbrace $, where $(w_i)_{i=1,\ldots, N_x}$ are the linear hat-functions with the property that $w_i(x)$ is a piecewise-linear function with $w_i(x_j)=\delta_{i,j}$ and $\delta_{i,j}$ is the  Kronecker delta function. Namely, we have

\begin{center}
\begin{tabular}{rl}
$\delta_{i,j}=$ & $ \left \{ \begin{array}{lrl}
1; \hspace{1cm} i=j,\\
0; \hspace{1cm} i\neq j,
          \end{array}\right.$ \\ \\
$w_i(x) = $ & $\left \{ \begin{array}{ll}
\displaystyle{ \frac{x-x_{i-1}}{h} }& \quad \forall \ x_{i-1}\leq x \leq x_i,\\
\displaystyle{\frac{x_{i+1}-x}{h} }& \quad \forall \ x_i\leq x \leq x_{i+1},\\
0 &\quad \forall \  x \in [0,1] \setminus [x_{i-1},x_{i+1}].
          \end{array}\right.$  \\
\end{tabular}
\end{center}

Now, we formulate the  semi-discrete problem as follows:\\
Find the functions $u_h$, $v_h$, $\varphi_h$ and $\psi_h$  such that
 $$u_h=\sum_{i=1}^{N_x} u_i(t) w_i(x),$$
$$\varphi_h=\sum_{i=1}^{N_x} \varphi_i(t) w_i(x),$$
$$\psi_h=\sum_{i=1}^{N_x} \psi_i(t) w_i(x),$$
 \hspace{0.2cm} and  $$v_h=\sum_{i=1}^{N_x} v_i(t)w_i(x).$$
 The discrete  boundary conditions read
      $$
     \varphi(x_{N_x},t)= \varphi(x_0,t)=0.
     $$
   $$\psi(x_{N_x},t)= \psi(x_0,t)=0.$$
Then, for
$$U(t)=\left[ u_h(x_0,t), \ldots, u_{h}(x_{N_x},t)\right]^t,$$
$$V(t)=\left[ v_h(x_0,t), \ldots, v_h(x_{N_x},t)\right]^t,$$
$$\Phi(t)=\left[ \varphi_h(x_0,t), \ldots, \varphi_{h}(x_{N_x},t)\right]^t,$$
and
$$\Psi(t)=\left[ \psi_h(x_0,t), \ldots, \psi_h(x_{N_x},t)\right]^t,$$
we have the following matrix formulation for the semi-discrete problem:
\begin{equation} \label{50000}
   \left \{ \begin{array}{lrl}
M \displaystyle{ \frac{d U}{dt} }= -K\Psi+S\Psi,\vspace{0.2cm} \\
M \displaystyle{\frac{d V}{dt}} =-K\Psi- S\Phi-M\Psi,\vspace{0.2cm} \\ 
\displaystyle{\frac{d \Phi}{dt}}=U(t),\vspace{0.2cm} \\
\displaystyle{\frac{d \Psi}{dt}}=V(t),
 \end{array}\right.
 \end{equation}
  where $M_{i,j}=(w_i,w_j)$ is the mass matrix, $K_{i,j}=(w_i',w_j')$ is the rigidity matrix and  $S$ is the matrix defined by  $S_{i,j}=(w_i',w_j)$. Note that $S$ is not symmetric.  \\
\subsection{Fully-discrete scheme in Finite Differences}

   We design an explicit unconditionally stable scheme using finite differences and  we consider  the classical  method of advancing the solution in time known as the leapfrog scheme. For the sake of completeness, we recall here the definition of this method.
   \begin{definition}(The leapfrog time scheme \cite[p.339]{D})\\
 Let $U$ denote a typical dependent variable, governed by an
equation of the form
\begin{equation}\label{abstract}
\frac{d U}{dt}= F(U).
\end{equation}
The continuous time domain $(0,T)$ is replaced by a sequence of
discrete moments $\lbrace 0,\Delta t, 2 \Delta t, \ldots, n \Delta t, \ldots \rbrace.$\\
The solution at these moments is denoted by $U^n = U(n \Delta t)$.
If this solution is known up to time $t = n\Delta t$, then the right-hand side of \eqref{abstract}, $F^n = F(U^n)$,  can be  computed.\\
The time derivative is approximated by a centered difference
derivative is approximated by a centered difference
$$ \frac{U^{n+1}-U^{n-1}}{2\Delta t}= F^n.$$
Thus, the forecast value $U^{n+1}$ may be computed from the
old value $U^{n-1}$ and the tendency $F^n$:
\begin{equation}\label{eq14}
U^{n+1} = U^{n-1} + 2\Delta t F^n.
\end{equation}
\end{definition}
    \begin{remark}
  For $n=0$, the equation \eqref{eq14} gives $$U^1=U^{-1}+2\Delta t F^0,$$  then, the initial condition $U^1$ cannot be obtained using the leapfrog scheme, so normally a simple non-centered step
  $$U^1=U^0+\Delta t F^0,$$  is used to provide the value at $t=\Delta t.$
   \end{remark}

    For that  purpose we introduce a time step $\Delta t > 0$ and we set $t^n = n \Delta t$.  Then, we introduce $U^n=U(n \Delta t)$ the discrete solution of the semi-discrete equations (\ref{50000}).
Our aim consists in finding the discrete solutions $(\Phi^n ,U^n ,\Psi^n ,V^n)$ which satisfy the following leapfrog  scheme:

\begin{equation} \label{Ms}
   \left \{ \begin{array}{lrl}
M \displaystyle{ \frac{U^{n+1}-U^{n-1}}{2\Delta t}} = -K\Phi^n+S\Psi^n,\\ \\
M \displaystyle{\frac{ V^{n+1}-V^{n-1}}{2\Delta t}} =-K\Psi^n- S\Phi^n-M\Psi^n,
 \end{array}\right.
 \end{equation}

 \begin{equation} \label{15}
   \left \{ \begin{array}{lrl}
   \displaystyle{\frac{\Phi^{n+1}-\Phi^{n-1}}{2\Delta t}} = U^n,\\ \\
 \displaystyle{\frac{ \Psi^{n+1}-\Psi^{n-1}}{2\Delta t}} =V^n.
  \end{array}\right.
 \end{equation}
The initial conditions  are simply obtained as follows:
 \begin{equation} \label{ci}
   \left \{ \begin{array}{lrl}
 V^0= \displaystyle{\frac{ \Psi^{1}-\Psi^{0}}{\Delta t}},\\ \\
 U^0  = \displaystyle{ \frac{\Phi^{1}-\Phi^{0}}{\Delta t}},
 \end{array}\right.
 \end{equation}
 and
 \begin{equation} \label{ci1}
   \left \{ \begin{array}{lrl}
   \Phi^1=\Delta t U^0 + \Phi^0,\\
\Psi^1=\Delta t V^0+\Psi^0,\\
U^1=U^0- \Delta t \left( M^{-1}K \Phi^0+M^{-1} S \Psi^0\right),\\
V^1=V^0 -\Delta t M^{-1} \left( K \Psi^0- \Delta t   M^{-1}S  \Phi^0- \Delta t \ M \Psi^0\right).
 \end{array}\right.
 \end{equation}
 Finally, we rewrite \eqref{Ms}-\eqref{15} as follows:
 \begin{equation}\label{nn}
  \left \{ \begin{array}{lrl}
   U^{n+1}= U^{n-1}-2\Delta t  \ \left( - M^{-1} K \Phi^{n}-  M^{-1} S \Psi^{n}  \right),\\
   V^{n+1}=V^{n-1}-2\Delta t \ \left(  M^{-1}K \Psi^{n} - M^{-1} S \Phi^n - I\Psi^n\right),\\
   \Phi^{n+1}=2\Delta t \  U^n +\Phi^{n-1},\\
   \Psi^{n+1}=2\Delta t \  V^n+\Psi^{n-1}.
 \end{array}\right.
 \end{equation}

%
\begin{remark}
We notice here that  the conservation or  the dissipation  of the discrete energy, as we will see later on in this article, gives a good indication on the stability of  the  proposed fully-discrete scheme for the one-dimensional Timoshenko system under consideration. A rigorous  convergence result for this scheme together with the stability and consistency will be studied in \cite{CH-prep}.

The proof of the convergence results will be based, in part,  on a similar proof by Cowsar et al. \cite{converpp} for a mixed method approximation applied to the wave equation.
\end{remark}

\subsection{Numerical tests}
In order to show the behavior of the numerical solution, we start by choosing the appropriate initial conditions which satisfy the Dirichlet boundary conditions (\ref{2}) and constitute a paired solution of  the undamped Timoshenko system  (\ref{22}).  More precisely, we set
$$\varphi_0(x_i)=cos(\frac{2\pi x_i}{L}),$$
$$\psi_0(x_i)=sin( \frac{ 2\pi x_i}{L});$$
see Fig. \ref{fig:1} for the behavior of the initial data and Fig. \ref{fig:2222} for the variation of the Timoshenko solution in space and time up to time $T=10$.


\begin{figure}
\includegraphics[width=4cm]{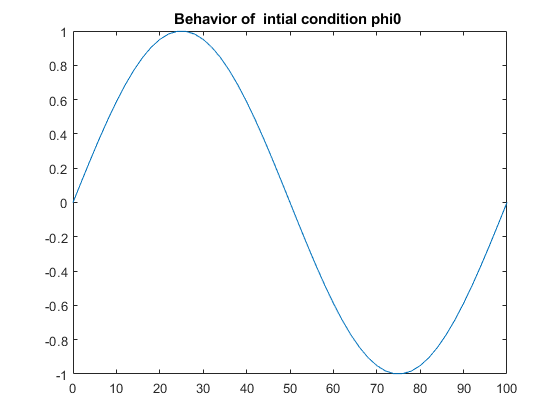}
\includegraphics[width=4cm]{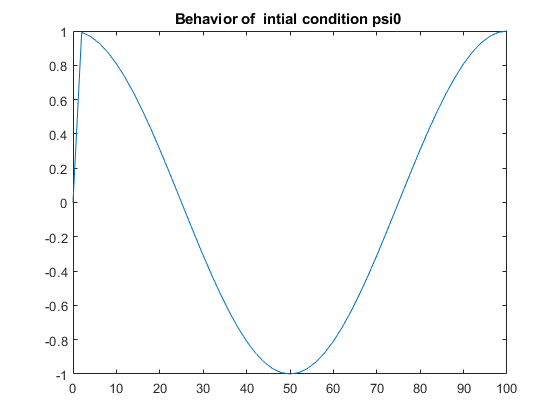}
\centering
\caption{The  initial conditions $\varphi_0$ and $\psi_0$.}
\label{fig:1}
\end{figure}

\begin{figure}
\includegraphics[width=5cm]{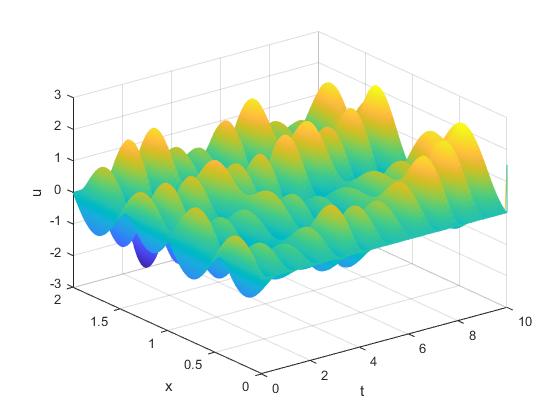}
\includegraphics[width=5cm]{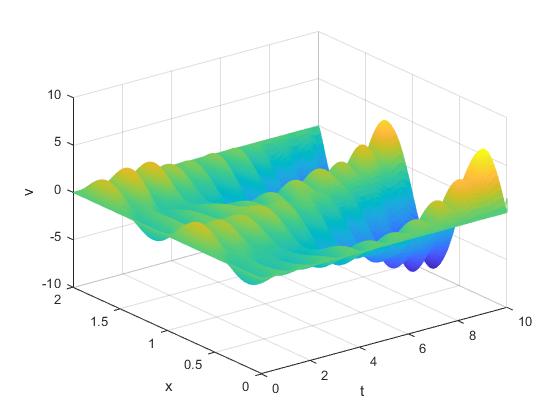}
\centering
\caption{The numerical behavior of the solutions $u(x_i,t_n)$, $v(x_i,t_n)$.  }
\label{fig:2222}
\end{figure}
%

\textbf{Comment 1.}  The numerical behavior of the solution of  (\ref{22}),  for  $N_x=50,$ $L=2$, $T=10$, $h=0.04$ and $\Delta t= c*h$ where  $c$ is a  positive constant, is obtained  in Fig. \ref{fig:2222}. We note that the study of the undamped case is just a first step in our numerical approach but it will allow us to reach later on the understanding of the damped case.  Thus, it would be interesting, besides the importance of the intrinsic study of the free wave equation, to compare and show the difference between the undamped case and the damped one.  \\

  As we know, the behavior of the solution influences the properties of the total energy $E(t)$ of the system, and from the figures above we can see that the  discrete  Timoshenko system  is purely conservative. This observation is substantiated by the next result as in Proposition \ref{prop1} below.\\

      Using the numerical scheme as previously described,   we present in the following proposition the first property related to the discrete energy of the system (\ref{002})-(\ref{122}).
\begin{proposition}(Conservation property of the discrete energy) \label{prop1}\\
Let $h>0 $ , $\Delta t>0 $  and   $(\Phi^n , U^n,\Psi^n , V^n)$ be the  solution of the  finite difference
scheme (\ref{Ms})-\eqref{15} associated with  initial conditions (\ref{ci})-(\ref{ci1}).\\
 Then, for all $n\in \lbrace 0, \ldots, N_t\rbrace,$ the discrete energy, defined by
\begin{equation}\label{ener}
E^n:= \frac{1}{2}\left(\| U^{n}\|_M^2+\|V^{n}\|_M^2+\|\Phi^{n}\|_K^2+\|\Psi^{n}\|_K^2 +\|\Psi^n\|_M^2\right),\ \ \forall\  n=\lbrace 0, \cdots, Nt\rbrace,
\end{equation}
 satisfies the following conservation property:
\begin{equation}\label{ene1}
\frac{1}{2\Delta t} \left( E^{n+1}-E^{n-1}\right)=0,
\end{equation}
here $\|.\|_M$ denotes the norm $\|u\|^2_M=(Mu,u).$
\end{proposition}

\begin{proof}
 To prove (\ref{ene1}) we use the energy method
in the following manner. We multiply $(\ref{Ms})_{1,2}$ by  the
discrete multipliers  $(U^{n+1}+U^{n-1})$, $(V^{n+1}+V^{n-1})$,
 we get  \\
 \begin{eqnarray}\label{aa2}
\frac{1}{2\Delta t}\left[\|U^{n+1}\|^2-\|U^{n-1}\|^2\right]&&=-(K\Phi^n,U^{n+1})-(K\Phi^n,U^{n-1})\nonumber \\
 &&+(S\Psi^n,U^{n+1})+(S\Psi^n,U^{n-1}),
\end{eqnarray}
 and
 \begin{eqnarray}\label{a2}
  \frac{1}{2\Delta t}\left[\|	V^{n+1}\|^2-\|V^{n-1}\|^2\right]&&=-(K\Psi^n,U^{n+1})-(S\Phi^n,V^{n-1})\nonumber \\
  &&+(S\Phi^n,V^{n+1})-(K\Psi^n,V^{n-1})\nonumber \\  &&-(S\Phi^n,V^{n-1})-(M\Psi^n ,V^{n+1}+V^{n-1}),
 \end{eqnarray}

 Second, we multiply the equations  $(\ref{15})_{1,2}$ by
$ K(\Phi^{n+1}+\Phi^{n-1})$,\\ $K(\Psi^{n+1}+\Psi^{n-1})$, respectively, we obtain
  \begin{equation}\label{a3}
  \frac{1}{2\Delta t}\left[\|\Phi^{n+1}\|_K^2-\|\Phi^{n-1}\|_K^2\right]=(U^{n},K\Phi^{n+1})+(U^n,K\Phi^{n-1})
 \end{equation}
 and
 \begin{equation}\label{a4}
  \frac{1}{2\Delta t}\left[\|\Psi^{n+1}\|_K^2-\|\Psi^{n-1}\|_K^2\right]=(V^{n},K\Psi^{n+1})+(V^n,K\Psi^{n-1})
 \end{equation}

 Now, summing the equations (\ref{aa2}),(\ref{a2}), (\ref{a3}) and (\ref{a4}) and by taking the sum over $l= \lbrace 0,\cdots,n \rbrace,$  we observe that
\begin{eqnarray}\label{a5}
&&\sum_{l=0}^n (V^{l},K\Psi^{l+1})+(V^l,K\Psi^{l-1})+(U^{l},K\Phi^{l+1}) +(U^l,K\Phi^{l-1})-(K\Psi^l,U^{l+1})\nonumber \\ &&-(K\Phi^l,U^{l-1}) -(K\Psi^l,V^{l-1})-(K\Psi^l,U^{l+1})=-(KU^n,\Phi^{n-1})+(KU^{n+1},\Phi^n)
\nonumber \\ &&-(KV^{n+1},\Psi^n)+(KV^{n+1},\Psi^{n-1}).
\end{eqnarray}

 Using (\ref{nn}), we have
\begin{equation}\label{b1}
  U^{n+1}= U^{n-1}+\varepsilon(\Delta t)
\end{equation}
\begin{equation}\label{b2}
 V^{n+1}= V^{n-1}+\varepsilon(\Delta t)
\end{equation}
where, $\varepsilon (\Delta t)\rightarrow 0$ when $ \Delta t\rightarrow 0$.\\

 	Taking in to account the equalities  (\ref{15}) (\ref{b1}) and (\ref{b2}), we deduce that the equality (\ref{a5}) is vanish for any $n\in \lbrace 0,...,Nt \rbrace.$\\
 	On the other hand, using (\ref{15})$$ \displaystyle{\frac{\Phi^{n+1}-\Phi^{n-1}}{2\Delta t}} = U^{n},$$
  we have\begin{equation}
(S\Psi^n,U^{n+1})+(\Phi^n,SV^{n+1})\simeq \frac{(S\Psi^{n+1},\Phi^{n+1})-(S\Psi^{n-1},\Phi^{n-1})}{2\Delta t}.
\end{equation}

Using  (\ref{15}), the expressions of $\Psi^{n+1}$ and $\Phi^{n+1}$ can be written as follows:
\begin{displaymath}
\Psi^{n+1}=2\Delta t V^n + \Psi^{n-1},
\end{displaymath}
\begin{displaymath}
\Phi^{n+1}=2\Delta t U^n +\Phi^{n-1},
\end{displaymath}
and consequently we have
\begin{equation} \label{02}
\begin{tabular}{c}
$ 	 (S\Psi^{n+1},\Phi^{n+1})-(S\Psi^{n-1},\Phi^{n-1}) = (2\Delta t \  SV^{n},\Phi^{n-1})
 	$ \\
 	 $+(S\Psi^{n-1},2\Delta t \ U^n)+ 4(\Delta t)^2(SV^n,U^n)=\varepsilon(\Delta t) \ \rightarrow \  0 \  \textnormal{as} \ \Delta t \rightarrow 0.$
 	\end{tabular}
 	\end{equation}
Using the expression of the discrete energy (\ref{ener}), we obtain the  conservation result (\ref{ener}). This ends the proof of Proposition \ref{prop1}.
\end{proof}

   \begin{figure}
\centering
\includegraphics[width=7cm]{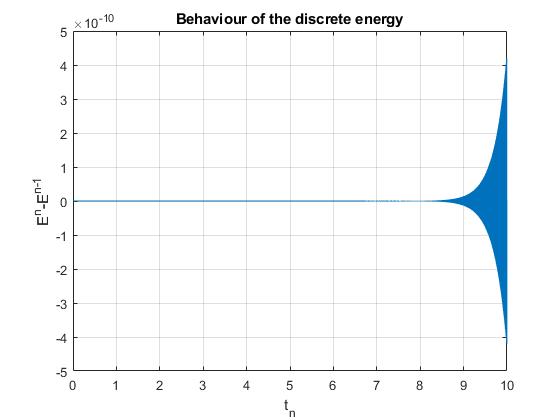}
\caption{The undamped case: the conservative property of the discrete energy $E^n$ (defined by (\ref{ener})).}
\label{fig:grosblabla}
\end{figure}

\textbf{Comment 2.} We observe that the energy difference $E^n-E^{n-1}$ is equal to zero  at almost every  time $t_n$. However, it is obvious  that Fig. \ref{fig:grosblabla} presents a peak of $5.10^{-5}$, but this peak does not have any impact on   the conservation result of energy since it is a digital zero due to the computational error. In conclusion, Figure \ref{fig:grosblabla} shows  the conservative character of the discrete  energy $E^n$ which is in agreement with the theoretical results.

\section{Exponential decay rate of the discrete energy }
\label{sec:2}
It is well-known that the energy associated with  the system (\ref{1}) can not be, in general, exponentially stable in case where we consider one damping term of the form $\psi_t$ in the second equation of (\ref{1}); see  \cite{13}. This exponential stability is only obtained in the case of equal-wave speeds  $\left( \frac{k}{\rho _{1}}=\frac{b}{\rho _{2}}\right)$. In this section,  we will confirm numerically this theoretical result. More precisely, we will prove that the damped Timoshenko system with a linear damping term considered in one equation  stumbles the exponential stability. Let us consider the following linearly damped Timoshenko system:
\begin{equation} \label{100}
     \left \{ \begin{array}{lrl}
  \varphi_{tt} -(\varphi_x+ \psi)_x=0,  \hspace{2.4cm } (x,t) \in (0,L)\times \mathbb{R}_+,\\
         \psi_{tt}-\psi_{xx}+(\varphi_x+\psi)+ \mu \psi_t=0, \hspace{0.5cm }(x,t) \in (0,L)\times \mathbb{R}_+,
          \end{array}\right.
       \end{equation}
 where  $\mu$ represents the damping coefficient. \\

        First, we will design a numerical simulation with finite element methods taking advantage of the discrete formulation carried out for the undamped case (\ref{1}).

      As we did for the undamped system (\ref{1}), we present here a matrix formulation for the semi-discrete linearly damped problem corresponding to \eqref{100} which reads as follows:
\begin{equation} \label{5555}
   \left \{ \begin{array}{lrl}
M \displaystyle{\frac{d U}{dt}} = -K\Psi+S\Psi,\vspace{.2cm} \\
M \displaystyle{\frac{d V}{dt} }=-K\Psi- S\Phi-M\Psi-\mu M V,\vspace{.2cm}  \\ 
\displaystyle{\frac{d \Phi}{dt}}=U(t),\vspace{.2cm}  \\
\displaystyle{\frac{d \Psi}{dt}}=V(t).
 \end{array}\right.
 \end{equation}
 Second, we consider the finite difference scheme applied to  (\ref{5555}) and we end up with  the following formulation. \\
 Find $(U^n,V^n,\Phi^n,\Psi^n)$ such that
    \begin{equation} \label{11115}
   \left \{ \begin{array}{lrl}
M\displaystyle{ \frac{U^{n+1}-U^{n-1}}{2\Delta t}} = -K\Phi^n+S\Psi^n,\vspace{.2cm} \\
M \displaystyle{ \frac{ V^{n+1}-V^{n-1}}{2\Delta t}} =-K\Psi^n-S\Phi^n-M\Psi^n-\mu M V^n,\vspace{.2cm} \\
   \displaystyle{\frac{\Phi^{n+1}-\Phi^{n-1}}{2\Delta t}} = U^n,\vspace{.2cm} \\
 \displaystyle{\frac{ \Psi^{n+1}-\Psi^{n-1}}{2\Delta t}} =V^n.
 \end{array}\right.
 \end{equation}
 Using the same arguments as in the previous section, we have the following result about the variation of the discrete energy (\ref{ener}).
 \begin{theorem}\label{th1}
Let $\Delta t>0$. Then, the discrete energy $E^{n+1}$,  associated with the solutions of the discrete equations (\ref{11115}) with the initial conditions (\ref{ci})-(\ref{ci1}) and the boundary conditions (\ref{BC}) and  defined by (\ref{ener}), verifies
\begin{equation}
\frac{1}{2\Delta t} \left( E^{n+1}-E^{n-1}\right)= -\mu  \|V^n\|^2_M \leq 0.
\end{equation}
Consequently, we obtain the energy dissipation law which reads as follows:
\begin{equation}\label{dissip1}
E^{n} \leq E^0,\;   \forall  \; n\in {0,\cdots, Nt}
\end{equation}
 \end{theorem}
 \begin{proof}
 Similarly to the continuous case, we use the techniques of multipliers at a discrete level given by $(U^{n+1}+U^{n-1})$, $(V^{n+1}+V^{n-1})$,   $ K(\Phi^{n+1}+\Phi^{n-1})$ and $K(\Psi^{n+1}+\Psi^{n-1})$  and we organize the results in order to obtain the difference $E^{n+1} - E^n$ as follows:
\begin{equation}
\begin{array}{l}
 \displaystyle \frac{1}{2\Delta t} [ \|V^{n+1}\|^2_M+\|\Psi^{n+1}\|^2_K-\|V^{n-1}\|^2_M-\|\Psi^{n-1}\|^2_K  +\|U^{n+1}\|^2_M \vspace{.2cm}\\
 \displaystyle -\|U^{n-1}\|_M^2 + \|\Phi^{n+1}\|_K^2
-\|\Phi^{n-1}\|^2_K ]+(M\Psi^n,V^{n+1}+V^{n-1}) \vspace{.2cm}\\
 \displaystyle=(\Phi^n , SV^{n+1}-SV^n)-(K\Phi^n, U^{n+1}+U^{n-1})
+(S\Psi^n ,U^{n+1}- U^{n-1}) \vspace{.2cm}\\
 +(K\Phi^{n+1},U^n) +(K \Phi^{n-1},U^n)
 	-\mu (MV^n,V^{n+1}+V^{n-1}).
 	\end{array}
 \end{equation}
Now, we use  the discrete energy expression (\ref{ener}),  we take the sum over $l=  0,\ldots,n$, and we use the initial conditions (\ref{ci1}) and the boundary conditions (\ref{BC}), we deduce that
 $$ \sum_{l=0}^{n}\frac{E^{l+1}-E^{l-1}}{2\Delta t} = \frac{E^{n+1}-E^{n-1}}{2\Delta t} = -\mu (M V^n,  V^{n})=-\mu \|V^n\|_M^2,$$
and hence
 \begin{equation}                                         \frac{E^{n+1}-E^{n-1}}{2\Delta t}=-\mu \|V^{n}\|^2_M.
   \end{equation}
  Then \eqref{dissip1} is straightforward.
   This ends the proof of Theorem \ref{th1}.
 \end{proof}

 Here, we show the numerical experiments related to the linearly damped Timoshenko system \eqref{100}. We consider for this system the same speeds of wave propagation and we obtain the variation of the discrete energy characterized by the exponential decay rate in terms of time; see Fig. \ref{fig:kiii}.
   \begin{figure}[h!]
\centering
\includegraphics[width=5cm]{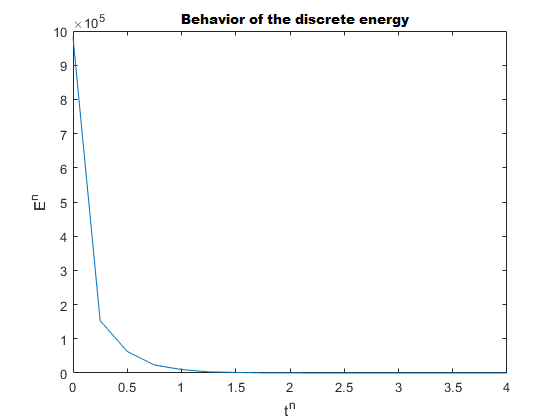}
\includegraphics[width=5cm]{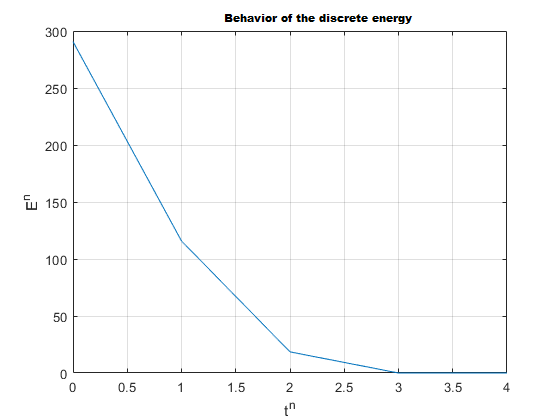}
\caption{The damped case: the discrete energy of
 (\ref{100}) expressed as a function of $t_n$ with two different initial data.}
 \label{fig:kiii}
\end{figure}

In Fig. \ref{fig:kiii},  the following initial data are taken
$$\varphi_0(x)=sin(N \pi x/L) \quad \text{and} \quad
\psi_0(x)=cos(N \pi x/L).$$
We also consider the following values for the  constants $L$, $T$, $h$  $c$, $N_x$ and $Nt$:
\begin{equation}\label{data}
L=50,
T=4,
c=0.2
N_x=10,
h= L/N_x,
k=c*h,
Nt=(N_x*T)/(c*L).
\end{equation}
  \begin{figure}[h!]
  \includegraphics[width=5cm]{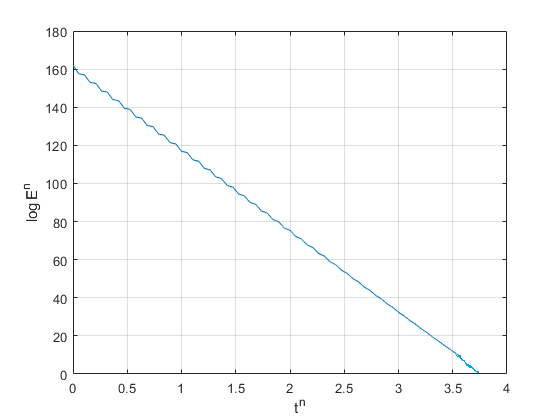}
\centering
\caption{Exponential decay rate of the discrete energy.}
\label{fig:kiki5}
\end{figure}

\textbf{Comment 4.} The energy $E^n$ decays like an exponential function $\exp(-\mu t^n)$ for $\mu >0$,  in the full damping case the discrete counterpart of the Timoshenko system is exponentially stable. \\
Fig. \ref{fig:kiki5} presents the graph of $log(E^n)$  as a function of $t^n$ which is in good agreement with the theoretical result. More precisely,  we have
 \begin{equation}
log(E^n)=-\mu t^n +b,  \hspace{0.2cm} -\mu \simeq -45.71, \hspace{0.2cm} b= 160,\hspace{1cm} \forall  \ n=1,\ldots,Nt,
\end{equation}
\begin{equation}
E^n=exp(b)\exp(-\mu t^n), \hspace{1cm} \forall \  n=1,\ldots,Nt.
\end{equation}
\section{Polynomial decay rate of the discrete energy}\label{sec:3}
\subsection{Nonlinear damping of type "$\vert s \vert s $"}\mbox{}\\
In this section, we deal with the following  nonlinear damped  Timoshenko system
 \begin{equation} \label{T1}
     \left \{ \begin{array}{ll}
  \varphi_{tt} - (\varphi_x+ \psi)_x=0, &  (x,t) \in (0,L)\times \mathbb{R}_+,\\
         \psi_{tt}-b\psi_{xx}+(\varphi_x+\psi)+  g(\psi_t)=0, & (x,t) \in (0,L)\times \mathbb{R}_+.
          \end{array}\right.
       \end{equation}

        The  description of the behavior of the energy corresponding to the solution of the system (\ref{T1}) is the goal of a great number of researchers. In \cite{ali}, the authors  consider a nonlinear vibrating Timoshenko system with thermoelasticity with second sound and used a fourth-order finite difference scheme to compute the numerical solutions.  Let us also recall that the question of stability  is strongly depending on the choice of the types of dissipation under consideration.\\

        In our case  we  first assume  that $g(\psi_t)=|\psi_t |\psi_t $, then in the second part of this section we will assume that $g(\psi_t)$ is given by $exp(\frac{-1}{(\psi_t)^2})$.
  In order to obtain the space discretization of the system (\ref{T1}), we use the Galerkin approximation as in the previous section. So  we first introduce the following new solutions:
   \begin{equation} \label{T3}
     \left \{ \begin{array}{lrl}
  \varphi_t=u,\\
  u_t=\varphi_{xx} -\psi_x,\\
  \psi_t=v,\\
  v_t=\psi_{xx}-\varphi_x-\psi-|\psi_t|\psi_t.
          \end{array}\right.
       \end{equation}
       Then, the system (\ref{T1}) can be rewritten as follows:
   \begin{equation}\label{25}
    \left \{ \begin{array}{lrl}
   M \displaystyle{\frac{du}{dt}}=-K\Phi +S\Psi,\vspace{.2cm} \\
   M\displaystyle{\frac{dv}{dt}}=-K\Psi-S\Phi-M\Psi - M |V|V,\vspace{.2cm} \\
   \displaystyle{\frac{d\Psi}{dt}}=V,\vspace{.2cm} \\
  \displaystyle{ \frac{d\Phi}{dt}}=U,
    \end{array}\right.
\end{equation}
where the term $"|V|V"$ reads
\begin{equation}
V(t)|V(t)|=\sum_{i=0}^{N_x} |v_i(t)| v_i(t) w_i(x).
\end{equation}
  The system (\ref{25}) can be written now as : find $(\Phi,U,\Psi,V)$ such that
 \begin{equation}
  \left \{ \begin{array}{lrl}
   M \displaystyle{\frac{dU}{dt}}=-K\Phi +S\Psi,\vspace{.2cm} \\
   M\displaystyle{\frac{dV}{dt}}=-K\Psi-S\Phi-M\Psi - M  V|V|,\vspace{.2cm} \\
   \displaystyle{\frac{d\Psi}{dt}}=V,\vspace{.2cm} \\
   \displaystyle{\frac{d\Phi}{dt}}=U.
    \end{array}\right.
 \end{equation}
 Now, we approximate $V^n$ by $\frac{1}{2}(V^{n+1}+V^{n-1})$, then the finite differences scheme yields the following problem.\\
 Find  $(\Phi^n,U^n,\Psi^n,V^n)$ such that
  \begin{equation}\label{44}
  \left \{ \begin{array}{ll}
   M \displaystyle{\frac{U^{n+1}-U^{n-1}}{2\Delta t}}=&-K\Phi^n  +S\Psi^n,\vspace{.2cm} \\
   M\displaystyle{ \frac{V^{n+1}-V^{n-1}}{2\Delta t}}=&-K\Psi^n-S\Phi^n-M\Psi^n \vspace{.2cm} \\
   &- \frac{1}{2} M  (V^{n+1}+V^{n-1})|V^n|,
      \end{array}\right.
 \end{equation}
 and
   \begin{equation}\label{233}
  \left \{ \begin{array}{lrl}
   \displaystyle{\frac{\Psi^{n+1}-\Psi^{n-1}}{2\Delta t}}=V^n,\vspace{.2cm} \\
  \displaystyle{ \frac{\Phi^{n+1}-\phi^{n-1}}{2 \Delta t}}=U^n.
    \end{array}\right.
 \end{equation}
  \begin{theorem}\label{th3}
   Let $\Delta t>0$. Then,  the energy $E^{n+1}$, defined by (\ref{ener}) and associated with the solutions of the discrete equations
\eqref{44}-\eqref{233} with the initial conditions (\ref{ci})-(\ref{ci1}) and the boundary conditions (\ref{BC}), satisfies the following conservation property:
\begin{equation}\label{enrg2}
\frac{1}{2\Delta t}  \left( E^{n+1}-E^{n-1}\right)= - \frac{1}{2} \|V^{n+1}+V^{n-1}\|_M^2 |V^n| \leq 0.
\end{equation}
Moreover, we have the discrete energy dissipation law,
$$ E^{n} \leq E^0.$$
  \end{theorem}
  \begin{proof}
  We multiply the  equations  (\ref{44})$_{1,2}$    by $(U^{n+1}+U^{n-1})$ and $(V^{n+1}+V^{n-1})$,  respectively, and we obtain
\begin{eqnarray*}\label{12364}
  \frac{1}{2\Delta t} [\|U^{n+1}\|^2_M-\|U^{n-1}\|^2_M] &=&-K(\Phi^n, U^{n+1})+(K\Phi^n,U^{n-1})   \nonumber\\ &&+ (S\Psi^n ,U^{n+1})- (S\Psi^n,U^{n-1}),
     \end{eqnarray*}
    and
    \begin{eqnarray*}\label{12364}	
 	 &\displaystyle \frac{1}{2\Delta t}&  [\|V^{n+1}\|^2_M-\|V^{n-1})\|^2_M]=-(K\Psi^n,V^{n+1}) +(K\Psi^n,V^{n-1})-(S\Phi^n,V^{n+1})\nonumber\\ &+&(S\Psi^n,V^{n-1}) -\frac{1}{2}(M\Psi^n,V^{n+1})
 	+ \frac{1}{2}(M\Psi^n,V^{n-1}) - \frac{1}{2}\|V^{n+1}+V^{n-1}\|^2_M |V^n|.\nonumber\\
 	 \end{eqnarray*}

  We multiply the equations  $(\ref{233})_{1,2}$  by $K (\Phi^{n+1}+\Phi^{n-1})$ and $K(\Psi^{n+1}-\Psi^{n-1})$, respectively, then  we have
  \begin{equation}
  \frac{1}{2 \Delta t}[ \|\Phi^{n+1}\|^2_M -\|\Phi^{n-1}\|^2_M]=(U^n,K\Phi^{n+1})+(U^n,K\Phi^{n-1}),
  \end{equation}
  and
  \begin{equation}
    \frac{1}{2\Delta t} [\|\Psi^{n+1}\|_M^2- \|\Psi^{n-1}\|^2_M]= (V^n,K\Psi^{n+1})+(V^n,K\Psi^{n-1}).
  \end{equation}
    
Taking the sum over $l=  0,\ldots,n $ and using the expression of the energy (\ref{ener}) and (\ref{02})  we deduce the estimate (\ref{enrg2}).\\
This completes the proof of Theorem \ref{th3}.
  \end{proof}

   The aim now is to show the behavior of the discrete  energy in terms of the discrete time variation. For that purpose, we write the system (\ref{T1})-(\ref{T3}) in a matrix form. Then, we perform the numerical simulation of the discrete energy. The numerical results  are exposed in Fig. \ref{fig:kiki}.
      \begin{figure}[h!]
  \includegraphics[width=5cm]{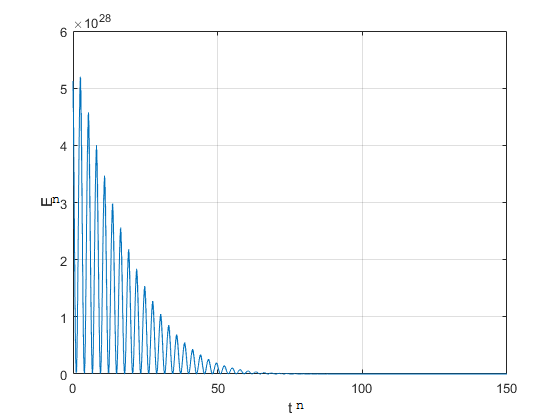}
\centering
\caption{ The behavior of the discrete energy $E^n$ in terms of the time step $t^n$.}
 \label{fig:kiki}
\end{figure}
\\
Fig. \ref{fig:kiki} shows  that the  decay of the discrete energy is  in this case slower than  the one obtained with a linear damping (see Fig.  \ref{fig:2222})  and here the lack of the exponential decay can be clearly identified. Instead of the exponential decay,  a typical polynomial profile for large time is in accordance with the analytical results established in the literature.

     Moreover, we express the behavior of $\log(E^n)$ as a function of $\log(t^n)$ as in the following numerical approximations.
     \begin{figure}[h!]
  \includegraphics[width=5cm]{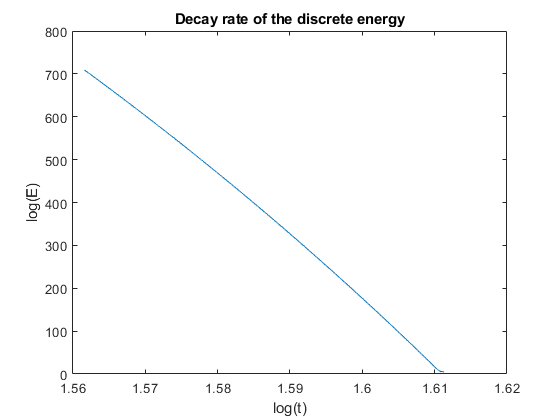}
\centering
\caption{Polynomial decay rate of the  discrete energy associated with the system (\ref{T1}): $\log(E^n)$ in terms of $\log(t^n)$. }
  \label{fig:kikok}
\end{figure}\\

\textbf{Comment 5} Based on the results obtained in Fig. \ref{fig:kikok} and using the constant values as in \eqref{data}, we can write $\log(E^n)= a_1 \log(t^n)+b_1$, with  $b_1=1.61$ and $a_1=-434.78$. Thanks to our numerical study we deduce an approximate value of the polynomial degree of the decay rate related to the discrete energy.
\subsection{Nonlinear damping  of the form $"\exp(-\frac{1}{s^2})"$}\mbox{}\\
In this subsection we consider the Timoshenko system subject to a nonlinear damping term which reads as follows:
 \begin{equation} \label{220}
     \left \{ \begin{array}{ll}
\varphi_{tt} -(\varphi_x+ \psi)_x=0,  &\hspace{1cm } (x,t) \in (0,L)\times \mathbb{R}_+,\\
         \psi_{tt}-(\psi_x)_x+(\varphi_x+\psi)+ g(\psi_t)=0, &\hspace{1cm }(x,t) \in (0,L)\times \mathbb{R}_+,
          \end{array}\right.
       \end{equation}
       where $g(x)=\exp\left(-1/x^2\right)$.\\

       This example is one of others that has been taken   to illustrate the optimal energy decay rate. More precisely, for this feedback, lower and upper energy estimates have been obtained, see \cite{alab} and \cite{art}. Here, the aim is to present some numerical results
completing thus the  theoretical results  already established.\\

  For the discrete  scheme, we will perform the same computations as previously done for the free wave equations. the only difference here is the nonlinear term that we discretize  with finite element method as follows. First, we have
  $$\psi_t=\sum_{i=0}^{N_x} \psi_i \  w'_i(x).$$
  Then, the nonlinear damping term is expressed as
  $$g(\psi_t)\simeq \sum_{i=0}^{N_x} g(\psi_i) \  w'_i(x).$$
Therefore,  the semi-discrete formulation reduces to looking for $(\Phi,U,\Psi,V)$ solution of
 \begin{equation}
  \left \{ \begin{array}{lrl}
   M\displaystyle{\frac{dU}{dt}}=-K\Phi +S\Psi,\vspace{.2cm} \\
   M \displaystyle{\frac{dV}{dt}}=-K\Psi-S\Phi-K g(\Psi),\vspace{.2cm}  \\
   \displaystyle{\frac{d\Psi}{dt}}=V,\vspace{.2cm}   \\
  \displaystyle{ \frac{d\Phi}{dt}}=U,
    \end{array}\right.
 \end{equation}
where, $g(\Psi)=(g(\psi_0),g(\psi_1),\ldots,g(\psi_{N_x}))$ is a vector  such that $g(\psi_0)=g(\psi_{N_x})=0$ due to the consideration of the homogeneous Dirichlet boundary conditions.

 After using the classical finite difference method for the discretization of the time derivative terms, we end up with the discrete form associated with the system (\ref{220}) together with the boundary conditions \eqref{BC} and the initial conditions (\ref{8}) which thus consists to find  $(\Phi^n,U^n,\Psi^n,V^n)$ such that

  \begin{equation}\label{44}
  \left \{ \begin{array}{lrl}
   M \displaystyle{\frac{U^{n+1}-U^{n-1}}{2\Delta t}}=-K\Phi^n  +S\Psi^n,\\ \\
   M\displaystyle{\frac{V^{n+1}-V^{n-1}}{2\Delta t}}=-K\Psi^n-S\Phi^n-M\Psi^n - K  g(\Psi^n),
      \end{array}\right.
 \end{equation}
   \begin{equation}\label{233}
  \left \{ \begin{array}{lrl}
 \displaystyle{  \frac{\Psi^{n+1}-\Psi^{n-1}}{2\Delta t}}=V^n,\\ \\
  \displaystyle{ \frac{\Phi^{n+1}-\phi^{n-1}}{2 \Delta t}}=U^n,
    \end{array}\right.
 \end{equation}
where $g(\Psi^n)=(g(\psi_0(t^n),g(\psi_1(t^n),\ldots ,g(\psi_{N_x}(t^n) ),$ for  $n=0,\ldots, N_t.$

      For this type of nonlinear damping,   the conclusion of the numerical study related to the discrete energy in this case is  presented in  Fig. \ref{fig:kiko}.
        \begin{figure}[h!]
  \includegraphics[width=7cm]{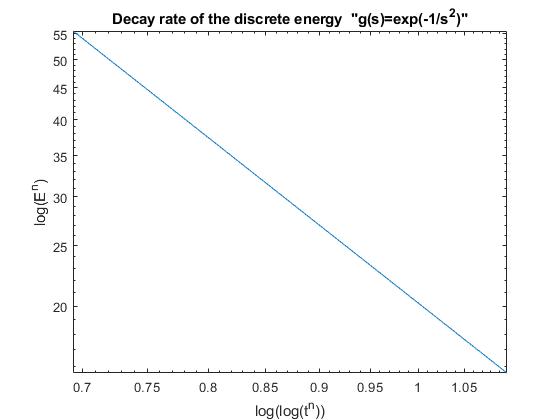}
\centering
\caption{Logarithmic decay rate of the discrete energy: $\log (E^n)$ as linear function of $\log(\log(t^n))$.}
 \label{fig:kiko}
\end{figure}\\
\textbf{Comment 6.} Figure \ref{fig:kiko}  shows a logarithmic  decay of the energy where we clearly obtained $\log (E^n)= a_2 \log(\log(t^n))+b_2,$ with $b_2=1.10$ and $a_2=-50$. Otherwise, the energy here has a logarithmic  decay, namely $E^n\simeq e^{b_2} (\log(t^n))^{a_2}$.

    \section{Conclusions and future work}\label{sec:4}
      In this work, we have addressed an important problem in mathematical analysis of beam, namely
 the problem of determining the  decay rate of the discrete energy  by taking into account a few dissipative mechanisms. As we already know,   Timoshenko system (\ref{1}) has two wave speeds  and we have proved numerically that is sufficient to consider only one dissipation mechanism  in order to obtain the exponential decay, for the case where the  speeds  are equal.
 Nevertheless, other dissipative cases have been considered in this article, we look for the behavior of the energy when the system is nonlinearly damped and we deduce an explicit (polynomial and logarithmic) decay rate of the discrete energy. One of the interesting futures of this work is the obtaining of  the approximate values of the constants (coefficients and monomial degrees) of the decay rate function in time of the discrete energy associated with the Timoshenko systems (\ref{1}), \eqref{100} and \eqref{T1} in an explicit manner.

 \section*{Acknowledgments}
 A part of this work was performed while the first author was visiting  LAMFA CNRS  UMR 7352  CNRS UPJV. The first author would like to thank all the LAMFA members for their hospitality and their help with warm thanks to Professor Olivier Goubet for many fruitful discussions.




\end{document}